\documentclass[12pt,fleqn]{amsart}

\usepackage{amssymb}
\usepackage{enumerate}

\newtheorem{theorem}{Theorem}[section]
\newtheorem{lemma}[theorem]{Lemma}

\theoremstyle{definition}
\newtheorem{definition}[theorem]{Definition}

\newtheorem{construction}[theorem]{Construction}

\theoremstyle{remark}
\newtheorem{remark}[theorem]{Remark}
\numberwithin{equation}{section}

\begin{document}

\title{Musing on Kunen's compact $L$-space}

%    author two information
\author[G.\ Plebanek]{Grzegorz Plebanek}
\address{Instytut Matematyczny, Uniwersytet Wroc\l awski}
\email{grzes@math.uni.wroc.pl}
\thanks{Partially supported by  the grant 2018/29/B/ST1/00223 from National Science Centre, Poland.}


%    \subjclass is required.
%\subjclass[2000]{Primary 46B26, 46E15; Secondary 46E27.}

%\date{}

%"Communicated by" -- provide editor's name; required.
%\commby{Nigel Kalton}
%\subjclass[2010]{Primary 46B26, 46B03, 46E15.}

%\maketitle

%%%%%%%%%%%%%%%%%%%%%%%%%%%%%%%%%%%%%%%%%LOCAL SHORTENINGS
%%%%%%%%%%%%%%NUMBERS
\newcommand{\con}{\mathfrak c}
\newcommand{\eps}{\varepsilon}
%%%%%%%%%%%%%%%%%%%%FRAK FAMILIES
\newcommand{\alg}{\mathfrak A}
\newcommand{\algb}{\mathfrak B}
\newcommand{\algc}{\mathfrak C}
\newcommand{\ma}{\mathfrak M}
\newcommand{\pa}{\mathfrak P}
%%%%%%%%%%%%%%%%%%%%SCRIPT FAMILIES
\newcommand{\BB}{\protect{\mathcal B}}
\newcommand{\AAA}{\mathcal A}
\newcommand{\CC}{{\mathcal C}}
\newcommand{\FF}{{\mathcal F}}
\newcommand{\GG}{{\mathcal G}}
\newcommand{\LL}{{\mathcal L}}
\newcommand{\NN}{{\mathcal N}}
\newcommand{\UU}{{\mathcal U}}
\newcommand{\VV}{{\mathcal V}}
\newcommand{\HH}{{\mathcal H}}
\newcommand{\DD}{{\mathcal D}}
\newcommand{\ZZ}{{\mathcal Z}}
\newcommand{\RR}{\protect{\mathcal R}}
\newcommand{\ide}{\mathcal N}
%%%%%%%%%%%%%%%%%%%%%%%SYMBOLS
\newcommand{\btu}{\bigtriangleup}
\newcommand{\hra}{\hookrightarrow}
\newcommand{\ve}{\vee}
\newcommand{\we}{\cdot}
\newcommand{\de}{\protect{\rm{\; d}}}
\newcommand{\er}{\mathbb R}
\newcommand{\qu}{\mathbb Q}
\newcommand{\supp}{{\rm supp} }
\newcommand{\card}{{\rm card} }
\newcommand{\wn}{{\rm int} }
\newcommand{\wh}{\widehat }
\newcommand{\ult}{{\rm ULT}}
\newcommand{\vf}{\varphi}
\newcommand{\osc}{{\rm osc}}
\newcommand{\ol}{\overline}
\newcommand{\me}{\protect{\bf v}}
\newcommand{\ex}{\protect{\bf x}}
\newcommand{\stevo}{Todor\v{c}evi\'c}
\newcommand{\cc}{\protect{\mathfrak C}}
\newcommand{\scc}{\protect{\mathfrak C^*}}
\newcommand{\lra}{\longrightarrow}
\newcommand{\sm}{\setminus}
\newcommand{\uhr}{\upharpoonright}
\newcommand{\en}{\mathbb N}
\newcommand{\sub}{\subseteq}
\newcommand{\ms}{$(M^*)$}
\newcommand{\m}{$(M)$}
\newcommand{\MA}{MA$(\omega_1)$}
\newcommand{\clop}{\protect{\rm Clop} }
\newcommand{\cf}{\protect{\rm cf} }
%%%%%%%%%%%%%%%%%%%%%%%%%%%%%%%%%%%%%%%%%%%%%%%%%%%%%%%%%%%

\begin{abstract}
We present a connected version of the compact $L$-space
constructed by Kenneth Kunen under CH.
We show that this provides a Corson compact space $K$ such that the Banach space $C(K)$ is
 isomorphic to no space of continuous function on a zero-dimensional compactum.
\end{abstract}

\maketitle

\section{Introduction}\label{intro}

Assuming the continuum hypothesis, Kunen \cite{Ku81} presented a construction of  a compact $L$-space --- a nonseparable space which is hereditarily Lindel\"of.
A compactified Suslin line  is also an $L$-space but its existence requires axioms stronger than CH.
Besides, Kunen's compactum has very interesting measure-theoretic properties which are not possible on linearly ordered spaces.

Saying that $\mu$ is a measure on a compact space $K$ we always mean that $\mu$ is a finite Borel measure on $K$ that is
inner-regular.
The basic idea that was behind Kunen's construction was that a compact space $K$ is an $L$-space whenever
it carries a probability measure $\mu$ such that

\begin{enumerate}[\bf M.1]
\item $\mu$ vanishes on singletons;
\item $\mu$ is strictly positive;%, that is $\mu(U)>0$ for every nonempty open set $U\sub K$;
\item $\mu(B)=0$ for every Borel set $B$ having empty interior;
\item every Borel set of measure zero set is metrizable.
\end{enumerate}

Note  that M.3 implies that $\mu(B)=\mu(\overline{B})$ for every Borel set $B$, and
we can check the required properties of $K$ as follows.

The space $K$ is not separable: For every countable $A\sub K$ we have $\mu(A)=0$ by M.1; hence,  $\mu(\overline{A})=0$ so $\overline{A}\neq K$.

To see that $K$ is hereditarily Lindel\"{o}f consider any family $\UU$ of open subsets of $K$.
Writing $U=\bigcup\UU$, we first take, using inner-regularity of $\mu$, a countable subfamily $\VV\sub\UU$ so that its union $V=\bigcup\VV$ satisfies $\mu(V)=\mu(U)$.
Hence $\mu(U\sm V)=0$ and $\mu(\overline{U\sm V})=0$. Consequently, by M.4,  $U\sm V$ is a subset of a compact metrizable space $\overline{U\sm V}$
so $U\sm V$ is covered by a countable subfamily $\VV_1\sub \UU$. This means that
$\VV\cup\VV_1$ forms  a countable subcover of $U$.

Around the same time, Haydon \cite{Ha78} and Talagrand \cite{Ta80}
 presented their constructions of compact spaces with measures carried out for different purposes.
 Those constructions, however,   shared some features of
the one from \cite{Ku81}. All the  three examples were later amalgamated by Negrepontis to the Kunen-Haydon-Talagrand example
solving a number of problems from topology and functional analysis, see \cite[section 5]{Ne84}.

\begin{definition}\label{nm}
A (regular Borel) measure $\mu$ on a compact space $K$ is {\em normal} if $\mu(B)=0$ for every Borel set $B$ with empty interior.
\end{definition}

With such a definition we can briefly say that Kunen constructed a dense-in-itself compact space  supporting a normal measure
such that  all the measure zero sets are metrizable.
Definition \ref{nm} follows the tradition in functional analysis  originated in   Dixmier's paper published  in 1951.
 Recall
that a normal measure on a compact space $K$ defines a so-called normal functional on the Banach lattice $C(K)$,
one which is order continuous, see \cite[section 4.7]{DDLS} for details.
Normal measures are sometimes called hyperdiffuse, see e.g.\ the survey paper by Flachsmeyer and Lotz \cite{FL80}.
We should warn the reader that Dixmier's notion of a normal measure has little to do with the same term used widely in set theory, in the context of real-valued measurable cardinals.

Kunen's compactum mentioned above  is defined as an inverse system of Cantor sets so it is zero-dimensional.
Recall that a typical example of a normal measure is the natural measure defined on the Stone space of the measure algebra $\alg$ of the Lebesgue measure $\lambda$ on $[0,1]$. Since the algebra $\alg$ is complete, its Stone space is extremally disconnected.

A couple of years ago, Garth Dales brought our attention to an old question, if
a compact connected space can carry a normal probability measure, see  \cite{FP64} and \cite{FL80}.
The problem was partially motivated by a result of Fishel and Papert \cite{FP64} who proved that
a compact space carries no nonzero normal probability measure whenever $K$ is {\em locally} connected.
Answering the question, we proved that there is a compact connected space of weight $\con$ admitting a normal probability measure,
using some Kunen-like inductive construction of length $\con$ carried out in ZFC. The preprint
\cite{Pl15} remained unpublished but the result was included  in the monograph \cite{DDLS} (as Theorem 4.7.24).

Kunen's ideas from \cite{Ku81} were further modified and developed in  a number of articles, see e.g.\
D\v{z}amonja and Kunen \cite{DK93,DK95}, Kunen and  van Mill \cite{KM95}
Brandsma and  van Mill \cite{BM98, BM00}, Kunen \cite{Ku08}, Borodulin-Nadzieja and  Plebanek
\cite{BNP16}, Dales and Plebanek \cite{DP19}.
In the present note we come back to the method of \cite{Ku81} once again to discuss the following result
(see the next section for all the unexplained terminology and notation)

\begin{theorem}\label{main}
Assuming $\cf(\NN)=\omega_1$, there is a compact space $K$ such that

\begin{enumerate}[(i)]
\item $K$ is connected;
\item $K$ is the support of  a normal  probability measure $\mu$;
\item every Borel set $B\sub K$ which is  $\mu$-null is metrizable;
\item $K$ is a Corson compact $L$-space.
\end{enumerate}

Moreover, one can assure that the measure $\mu$ in question is either of countable Maharam type or of type $\omega_1$.
\end{theorem}

 The fact that the assumption of CH can be relaxed to
the present one, on the cofinality of the null ideal, was already noted by Kunen and van Mill \cite[Theorem 1.2]{KM95}.
It is worth recalling that hardly any feature of the space $K$ as in the theorem above is possible in the usual set theory.
Indeed, under Martin's axiom and the negation of CH

\begin{enumerate}[--]
\item there are no compact $L$-spaces (Juh\'asz \cite{Ju80});
\item  every measure on a Corson compact space has a metrizable support,
(see \cite{Ne84} and \cite{KM95});
\item no first-countable compact space admits a normal probability measure
(Zindulka \cite{Zi00}).
\item every measure on a first-countable compact space is of type $\omega$ (Plebanek \cite{Pl95}).
\end{enumerate}

The fact that, depending on the shape of
successor steps of the construction, the measure $\mu$ as in \ref{main} can be either of type $\omega$ or $\omega_1$ was
already  noted in \cite{Ku81}. The additional property that the resulting space is connected is the main point here.
As we shall see it comes as a result of  a rather straightforward modification of
Kunen's argument.
However, we give below a self-contained proof of Theorem \ref{main} because
we later analyse further properties of the  space $K$ resulting from our construction and show that $K$ gives rise to an interesting Banach space $C(K)$; see the final section.

 I wish to thank H.\ Garth Dales for our  discussion concerning the
 role of normal measures in Banach space theory.

\section{Preliminaries}

\subsection{Compact spaces and measures}
As we have already mentioned, by a measure $\mu$  on a compact space $K$ we always mean
a finite Radon measure, that is
a measure defined on the Borel $\sigma$-algebra $Bor(K)$ which satisfies, for every $B\in Bor(K)$,
the regularity condition
 \[ \mu(B)=\sup\{\mu(F): \overline{F}=F\sub B\}.\]
We say that the measure $\mu$ is strictly positive on $K$ or that $K$ is the support of $\mu$ if $\mu(U)>0$ for every
nonempty open set $U\sub K$.

Recall that if $\mu$ is a measure on $K$ (vanishing on points) then the Maharam type of $\mu$ can be defined
as the density character of the corresponding Banach space $L_1(\mu)$ or, equivalently,
the density character of the underlying measure algebra with respect to the Fr\'echet-Nikodym metric.
Thus $\mu$ is of type $\omega$ if there is a countable family $\CC\sub Bor(K)$ such that
\[ \inf \{\mu(B\btu C): C\in\CC\}=0,\]
for every $B\in Bor(K)$.

Let $\lambda$ be the Lebesgue measure on $[0,1]$ and let
 $\NN$ denote the $\sigma$-ideal of Lebesgue null sets.
 Then  $\cf(\NN)$ stands for its cofinality, so
$\cf(\NN)=\omega_1$ amounts to saying that there is a family $\{N_\xi:\xi<\omega_1\}\sub\NN$ such that
for every $N\in\NN$ there is $\xi<\omega_1$ with $N\sub N_\xi$.
Recall that basic  cardinal invariants of Radon measures
are determined by their corresponding measure algebras. In particular,
if $\mu$ is a measure on a metrizable compact space $K$ and $\mu$ vanishes on singletons then the cofinality of its null ideal is
equal to $\cf(\NN)$, see Fremlin \cite{Fr89}.

We shall use the following consequence of a result due to Cicho\'n, Kamburelis and Pawlikowski \cite{CKP85}
(see also \cite[Proposition 6.9]{Fr89}).

\begin{theorem}\label{ckp}
If $\cf(\NN)=\omega_1$ then for every measure $\mu$ on a compact metrizable space $K$ there
is a family $\{F_\xi:\xi<\omega_1\}$ of closed nonnegligible sets such that  every $B\in Bor(K)$ with $\mu(B)>0$
contains $F_\xi$ for some $\xi<\omega$.
\end{theorem}

A zero set in a topological space $X$ is one of the form $f^{-1}[\{0\}]$ where
the function $f:X\to \er$ is continuous.
It is easy to check that a subset of a compact space is a zero set if and only if it is closed and $G_\delta$.
We write $\ZZ(X)$ for the family of all closed $G_\delta$ subsets of $X$.
Note that if a compact space $K $ is an $L$-space then it is perfectly normal, i.e.\
every  closed subset $F$ of it is $G_\delta$ so $F\in\ZZ(K)$.

\begin{lemma}\label{1:1}
Let $K$ be a compact space, and suppose that $\mu$ is a strictly positive measure on $K$ such that
$\mu(Z)=0$ for every $Z\in\ZZ(K)$ with empty interior. Then $\mu$ is a normal measure.
\end{lemma}

\begin{proof}
 The assertion follows from the following observation.
\medskip

{\sc Claim.} Every closed set $F\sub K$ with empty interior is contained in some $Z\in\ZZ(K)$ with empty interior.
\medskip

To verify the claim, consider a maximal family $\FF$ of continuous functions $K\to [0,1]$ such that $f|F=0$ for $f\in\FF$ and $f\cdot g=0$ whenever $f,g\in\FF$, $f\neq g$.
Then $\FF$ is necessarily countable because $K$, being the support of a measure, satisfies the countable chain condition. Write $\FF=\{f_n: n\in\en\}$ and let
$f=\sum_n 2^{-n}f_n$ and  $Z=f^{-1}[0]$. Then the function $f$ is continuous so  $Z\in\ZZ(K)$. We have $Z\supseteq F$ and the interior of $Z$ must be empty by the maximality of $\FF$.
\end{proof}

If $f:K\to L$ is a continuous map and $\mu$ is a measure on $K$ then the image measure $f[\mu]$ on $L$  is defined by
\[ f[\mu](B)=\mu(f^{-1}[B]) \mbox{  for every Borel set } B\sub L.\]

\subsection{Inverse systems}\label{is}
We shall consider inverse systems of compact (metrizable) spaces with probability measures of the form
\[\langle K_\alpha,\mu_\alpha, \pi^\alpha_\beta: \beta<\alpha<\omega_1\rangle,\]
where   for all $\gamma <\beta<\alpha<\omega_1$

\begin{enumerate}[\bf \ref{is}(a)] \label{1:2}
\item $K_\alpha$ is a compact space and $\mu_\alpha$  is a probability measure on  $ K_\alpha$;
\item $\pi^\alpha_\beta:K_\alpha\to K_\beta$ is a continuous surjection;
\item  $ \pi^\beta_\gamma  \circ \pi^\alpha_\beta =\pi^\alpha_\gamma$;
\item $\pi^\alpha_\beta[\mu_\alpha]=\mu_\beta$.
\end{enumerate}

The following summarises basic facts on such inverse systems, see  \cite[Proposition 4.1.15]{DDLS} or \cite[0.33]{Ne84}.

\begin{theorem} \label{2:is}
Let $K$ be the limit of the system satisfying \ref{is}(a)--(d) with uniquely defined continuous surjections $\pi_\alpha:K\to K_\alpha$ for
$\alpha<\omega_1$.

\begin{enumerate}[(a)]
\item $K$ is a compact space and $K$ is connected whenever  all the space $K_\alpha$ are connected.
\item There is the unique probability measure $\mu$ on $K$ such that $\pi_\alpha[\mu]=\mu_\alpha$ for $\alpha<\omega_1$.
\item If every $\mu_\alpha$ is strictly positive then $\mu$ is strictly positive.
\end{enumerate}
\end{theorem}

We also need  the following standard fact.

\begin{lemma}\label{2:is2}
Let $K$ be the limit of the system satisfying \ref{is}(a)--(c).
Then every continuous function $f:K\to\er$ can be written as $f=g\circ \pi_\alpha$ for some $\alpha<\omega_1$
and a continuous function $g:K_\alpha\to\er$.

Consequently, for every zero set $Z\sub K$
 there are $\alpha<\omega_1$ and $Z'\in\ZZ(K_\alpha)$ such that
$Z=\pi_\alpha^{-1}[Z']$.
\end{lemma}

\begin{proof}
The first statement is an immediate consequence of the Stone-Weierstrass theorem;  the second statement clearly follows.
\end{proof}

\subsection{Corson compact spaces}
A compact space $K$ is {\em Corson compact} if, for some $\kappa$,  it can be embedded into the space
$\Sigma(\er^\kappa)$ of those $x\in\er^\kappa$ which have only countably many nonzero coordinates.
The class of Corson compacta is tightly connected to several aspects of Banach space theory, see \cite{Ne84}.

\section{A connected  $L$-space}

The following lemma describes the essential part of the inductive construction leading to \ref{main}.

\begin{lemma}\label{basic}
Let $K$ be a metrizable compact connected space, and let $\mu$ be a strictly positive measure on $K$.
If $F\sub K$ is a closed set with $\mu(F)>0$, then there
are a compact connected space $\wh{K}\sub K\times [0,1]$ and two strictly positive measures $\nu_1,\nu_2$ on
$\wh{K}$ such that, writing $\pi :\wh{K}\to K$ for the projection, the following are satisfied

\begin{enumerate}[(i)]
\item $\wn((\pi^{-1}[F])\neq\emptyset$;
\item $\pi[\nu_1]=\pi[\nu_2]=\mu$;
\item there is $B\in Bor(\wh{K})$ such that $\inf\{\nu_1(B\btu \pi^{-1}[A]): A\in Bor(K)\}>0$;
\item $\inf\{\nu_2(B\btu \pi^{-1}[A]): A\in Bor(K)\}=0$ for every $B\in Bor(\wh{K})$.
\end{enumerate}
\end{lemma}

\begin{proof}
Let $F_0$ be the support of $\mu$ restricted to $F$, that is
\[ F_0=F\sm\bigcup\{U: U\mbox{ open and } \mu(F\cap U)=0\}.\]
Let $\wh{K}=\{(x,t)\in K\times [0,1]: x\in F_0\mbox{ or } t=0\}$. Then
$\wh{K}$ is clearly a compact connected subspace of $K\times [0,1]$.
Moreover, the set $\pi^{-1}[F]$ contains $F_0\times [0,1]$, a set with non-empty interior so $(i)$ is granted.

We define the required measure $\nu_1$ on $\wh{K}$  by setting
\[\nu_1(B)= \mu\otimes\lambda (B),\]
for Borel sets $B\sub F_0\times [0,1]$, where $\lambda$ is the Lebesgue measure on $[0,1]$.
On the remaining part $(K\sm F_0)\times\{0\}$ we just copy the measure $\mu$ from $K\sm F_0$.
It is easy to verify $\pi[\nu_1]=\mu$ and that $(iii)$ holds for $B=F_0\times [1/2,1]$.

To define $\nu_2$ we choose a countable base $\UU$ of the space $F_0$ and fix an enumeration
$\{(U_n,q_n): n\in\omega\}$ of all pairs $(U,q)$ where $\emptyset\neq U\in\UU$ and $q\in(0,1)\cap\qu$.
Then we inductively construct a pairwise disjoint sequence of closed sets $L_n\sub F_0$ such that
$\mu(L_n\cap U_n)>0$ for every $n$. Note that it is easy to carry out such a construction under the
inductive assumption that ever $L_n$ has empty interior in the space $F_0$.

Now we take a function $h:F_0\to \wh{K}$ defined  by $h(x)=(x,q_n)$ for $x\in L_n$ (and $h(x)=(x,0)$ otherwise).
Set $\nu_2$ to be the obvious copy of $\mu$ outside $F_0\times [0,1]$ and put
$\nu_2(B)=\mu(h^{-1}[B])$ for $B\sub F_0\times [0,1]$. Then $\nu_2$ is strictly positive (by the way $h$ is defined)
and $(iv)$ holds since we can take $A=h^{-1}[B]$ there.
\end{proof}

\begin{remark}\label{r1}
We shall later call on the following additional property of $\pi:\wh{K}\to K$ defined as in the proof above.
Namely, $\pi$ is monotone, that is $\pi^{-1}[F]$ is connected for every connected $F\sub K$.

We shall also refer to the following. Take a continuous function $r$ on $[0,1]$ with $r(0)=0=\int_0^1 r(t)\; {\rm d}t$.
If we write $p$ for the projection $\wh{K}\to [0,1]$ onto the second coordinate then for any continuous function $f$ on $K$
we have
\[ \int_{\wh{K}} (r\circ p) \cdot (f\circ\pi) \; {\rm d}\nu_1=\int_{F_0 \times [0,1]} (r\circ p) \cdot (f\circ\pi) \; {\rm d}\mu\otimes\lambda=0,\]
 by the Fubini theorem.
\end{remark}

\begin{construction}\label{construction}
Let $K_0=[0,1]$ and let $\mu_0$ be the Lebesgue measure on $[0,1]$. Construct an inverse system of compact metrizable
spaces with measures
as in \ref{1:2}, with the following bookkeeping.

Given the space $K_\alpha$ with the measure $\mu_\alpha$,
using $\cf(\NN)=\omega_1$ fix an enumeration $\{N^\alpha_\xi: \xi<\omega_1\}$ of
family of $\nu_\alpha$-null subsets of $K_\alpha$ that is   cofinal the corresponding null ideal.
Moreover, using Theorem \ref{ckp}, fix a family $\{F^\alpha_\xi: \xi<\omega_1\}$ of closed subsets
of $K_\alpha$ such that $\mu_\alpha(F^\alpha_\xi)>0$ for every $\xi$ and every $B\in Bor(K_\alpha)$
of positive measure contains $F^\alpha_\xi$ for some $\xi<\omega_1$.

Fix also a bijection  $\phi: \omega_1 \times\omega_1\to \omega_1$ such that $\phi(\beta,\xi)=\alpha$ implies
$\beta\le\alpha$ --- $\vf$  will tell us that at step $\alpha$ we should take care of the set $F^\beta_\xi$, where
$\phi(\beta,\xi)=\alpha$.

There is nothing to do at limit step $\gamma$, we simply define $K_\gamma$ to be the inverse limit of the spaces
defined so far. Given $K_\alpha$ and the measure $\mu_\alpha$, we construct $K_{\alpha+1}$ as follows.
The set
\[ N=\bigcup_{\xi,\beta<\alpha} (\pi^\alpha_\beta)^{-1}[N^\beta_\xi],\]
is a countable union of $\mu_\alpha$-null sets so $\mu_\alpha(N)=0$.
For $\beta$ and $\xi$ such that  $\phi(\beta,\xi)= \alpha$, we consider the set
\[ H=(\pi^\alpha_\beta)^{-1}[F^\beta_\xi],\]
which satisfies $\mu_\alpha(H)>0$ so there is  a closed set $F\sub H\sm N$ with $\mu_\alpha(F)>0$.
Using Lemma \ref{basic} for such $F\sub K_\alpha$ we put $K_{\alpha+1}=\wh{K_\alpha}$ and  $\mu_{\alpha+1}=\nu_1$.
\end{construction}

\begin{proof}(of Theorem \ref{main})
We shall prove that  the limit $K$ of the inverse system as in \ref{construction} equipped with
the measure $\mu$ of type $\omega_1$ and both are as required.
We already know that  $\mu$  is strictly positive on $K$.

To prove that $\mu$ is a normal measure it is sufficient, by Lemma \ref{1:1},
to check that $\mu(Z)>0$ implies that $Z$ has nonempty interior for every zero set $Z$. But such a set $Z$
is of the form $Z=\pi_\beta^{-1}[Z']$ for some $\beta<\omega_1$ and a closed set $Z'\sub K_\beta$ (see  Lemma \ref{2:is2}).
If $\mu(Z)>0$ then $\mu_\beta(Z')>0$ so there is $\xi<\omega_1$ such that $F^\beta_\xi \sub Z'$.
Then at step $\alpha=\vf (\beta,\xi)$ we took care to
assure that the set
\[  (\pi_\alpha^{\alpha+1})^{-1}[(\pi^{\alpha}_\beta))^{-1} F^\beta_\xi]=(\pi^{\alpha+1}_\beta)^{-1}[F^\beta_\xi]\]
 gets nonempty interior. Since
 $\pi_\beta=\pi^{\alpha+1}_\beta\circ \pi_{\alpha+1}$ and $Z=\pi_\beta^{-1}[Z']$,
 it follows that $Z$ has nonempty interior too.

Once we know that $\mu$ is a normal
strictly positive measure on $K$ (clearly, $\mu$ vanishes on singletons), we conclude that $K$ is an $L$-space using
the argument from the introduction. Recall that this means that every closed subset of $K$ is a zero set.

Take a Borel set $B\sub K$ such that $\mu(B)=0$; then $\mu(\overline{B})=0$ (by normallity of $\mu$)
and $Z=\overline{B}\in \ZZ(K)$.
It follows that   $Z=\pi_\beta^{-1}[Z']$ for some $\beta$  and $Z'\sub N^\beta_\xi$ for some $\xi<\omega$.
 The set  $Z'$ remains unsplit from step $\alpha=\vf(\beta,\xi)$  so $\pi_\alpha$ is one-to-one on $Z$, i.e.\
  $Z$ is homeomorphic to the metric compactum $(\pi^\alpha_\beta)^{-1}[Z']$.

To see that  the space $K$ is Corson compact note that  if $x\in K$ then $t=\pi_0(x)\in [0,1]$ belongs to some
$N^0_\xi$ and
 it follows that all  coordinates of $x$ above $\vf(0,\xi)$ are zero.

The fact that we used the measure $\nu_1$ from Lemma \ref{basic} in the construction implies that
$\mu$ is indeed of type $\omega_1$, as \ref{basic}(iii) implies that no countable family
of Borel subsets of $K$ can form a $\btu$-dense family in $Bor(K)$.

To get another space $K$ supporting  a normal measure of countable type
 we carry out the whole of \ref{construction}
replacing $\nu_1$ by  $\nu_2$ at each successor step when referring to Lemma \ref{basic}.
Then Lemma \ref{basic}(iv)  guarantees that the resulting measure $\mu$  has type $\omega$ -- its measure algebra is be the same
as that of the Lebesgue measure.
\end{proof}

We can now point out an  interesting property  of our connected version of Kunen's $L$-space.

\begin{theorem}\label{prop}
Let $K$ be the compact space resulting from Construction \ref{construction}.
If a compact zero-dimensional space $Y$ is a continuous image of a closed subspace $Z$ of $K$ then
$Y$ is metrizable.
\end{theorem}

\begin{proof}
Let $g:Z\to Y$ be a continuous surjection from a closed set $Z\sub K$ onto a zero-dimensional space $Y$.
Since $K$ is perfectly normal, $Z$ is a zero set in $K$ so
$Z=\pi_\alpha^{-1}[Z']$ for some $\alpha<\omega_1$ and a closed set $Z'\sub K_\alpha$.
Using Remark \ref{r1} we note the following.
\medskip

\noindent {\sc Claim.} The set $\pi_\alpha^{-1}[\{t\}]$ is connected for every $t\in Z'$.
\medskip

As $Y$ is zero-dimensional, Claim means that $g$ is constant on every coset $\pi_\alpha^{-1}[\{t\}]$. Thus, we can
define a mapping $g':Z'\to Y$ by $g'(t)=g(x)$ where $\pi_\alpha(x)=t$. Then $g=g'\circ\pi_\alpha$ and this
implies that $g'$ is a continuous surjection from $Z'$ onto $Y$. As $Z'$ is compact metrizable, so is $Y$.
\end{proof}

Note that Theorem \ref{prop} says in particular that the space $K$ in question contains no
nonmetrizable zero-dimensional subspaces.
Recall that such a  feature is somewhat delicate; only quite recently Koszmider \cite{Ko16}
constructed  a ZFC example of a nonmetrizable compact space having that property.
Assuming the existence of a Lusin set, Marciszewski \cite{Ma20} gives another example
of that kind which is an Eberlein compact space.

\section{On the Banach space $C(K)$}

Given a compact space $K$, $C(K)$ denotes the Banach space of continuous real-valued functions on $K$
with the usual supremum norm. Recall that if two such Banach spaces $C(K)$ and $C(L)$ are isometric then
by the classical Banach-Stone theorem the underlying compacta $K$ and $L$ must be homeomorphic.
The situation changes dramatically if we ask for which pairs $K$ and $L$ the Banach spaces $C(K)$ and $C(L)$ are merely isomorphic, that this there is a linear surjection $T: C(K)\to C(L)$ such that
$m\cdot \|g\|\le \| Tg\|\le M\cdot \| g\|$ for some $m,M>0$ and every $g\in C(K)$; see
\cite{Pl15i} for further information and basic references to the subject.

By the classical Miljutin theorem, $C([0,1])$ and $C(2^\omega)$ are isomorphic as Banach spaces.
It was  along standing problem if for every compact space $K$ there is a zero-dimensional compact space $L$
such that $C(K)$ and $C(L)$ are isomorphic as Banach spaces. The first counterexample was obtained by
Koszmider \cite{Ko04} who used his involved technology of producing Banach spaces of continuous functions with `few operators'.
Another counterexample was presented by Avil\'es and Koszmider \cite{AK13}, as a by-product of their solution to
Namioka's problem on the class of Radon-Nikodym compacta.
We show below that our connected version of Kunen's $L$-space provides the third example of that kind.

Theorem \ref{final} given below will be derived from  our main result from \cite{Pl15}. Recall that by a $\pi$-base of a topological space
we mean a family $\VV$ of nonempty open sets having the property that every nonempty open set contains
some $V\in\VV$.

\begin{theorem}\label{pl15}
Let $K$ and $L$ be compact spaces such that $C(K)$ is isomorphic to $C(L)$.
Then $K$ has a $\pi$-base $\VV$ such that for every $V\in\VV$, $\overline{V}$ is a continuous image of some compact subspace
of $L$.
\end{theorem}

At this point,  we have to refer to \ref{construction} once again to check the following.

\begin{lemma}\label{nonsep}
Let $K$ be the compact space resulting from Construction \ref{construction}. Then
for every sequence of probability measures $\nu_n$ on $K$ there
is a nonzero continuous function $f:K\to\er$ such that $\int_K f\; {\rm d}\nu_n=0$ for every $n$.
\end{lemma}

\begin{proof}
Take any nonzero function $r:[0,1]\to [0,1]$ such that
$r(0)=0$ and $\int_0^1 r(t)\; {\rm d} t=0$ and
consider the family of functions $g_\alpha=r\circ p_\alpha\circ \pi_{\alpha+1}$ on $K$,
where, for $\alpha<\omega_1$, $p_\alpha:K_{\alpha+1}\to [0,1]$ denotes the restriction of the projection onto the second coordinate.
It is enough to check that for every measure $\nu$ on $K$ we have $\int_K g_\alpha\; {\rm d}\nu=0$
for all but countably many $\alpha<\omega_1$.

If $\nu$ is singular  with respect to $\mu$
then $\nu$ is concentrated on a metrizable closed  set $Z\sub K$.
Since $K$ is perfectly normal, $Z$ is a zero set in $K$ so
for some $\beta<\omega_1$ we have $Z=\pi_\beta^{-1}[Z']$ and
$Z'\sub N^\beta_\xi$ for some $\xi<\omega_1$.
It follows that for any $\alpha\ge \vf(\beta,\xi)$ we have $g_\alpha=0$ $\nu$-almost everywhere so
$\int_K g_\alpha\; {\rm d}\nu=0$.

On the other hand, if $\nu$ is a measure on $K$ that is absolutely continuous with respect to $\mu$ then
 $\nu(\cdot)=\int_{(\cdot)} h\; {\rm d}\mu$ for some measurable  $h:K\to\er$. Then there is a sequence of
 continuous functions $f_j:K\to\er$  converging to $h$ $\mu$-almost everywhere.
Hence, in such a case, it  remains to check that $\int g_\alpha\cdot f\; {\rm d}\mu=0$ for any continuous function $f$ and
all but countably many $\alpha$.
This follows from the fact that we have $f=f'\circ\pi_\beta$ for some $\beta<\omega$
and then $\int g_\alpha\cdot f\; {\rm d}\mu=0$ for any $\alpha\ge \beta$ by stochastic independence of
the function $f$ and $g_\alpha$, see Remark \ref{r1}.

The general case follows, as every measure $\nu$ on $K$ can be decomposed as $\nu=\nu'+\nu''$,
where $\nu'$ is singular and $\nu''$ is absolutely continuous with respect to $\mu$.
\end{proof}

We are now ready for the main result of this section.

\begin{theorem}\label{final}
Assuming $\cf(\NN)=\omega_1$, there is a Corson compact space $K$ such that
the Banach space $C(K)$ is isomorphic to no space of the form $C(L)$ with $L$ zero-dimensional.
\end{theorem}

\begin{proof}
We shall prove that the space from Theorem \ref{main} which is an inverse limit of the system
\ref{construction} supporting the measure $\mu$ of uncountable type is
as required.

Suppose  that $L$ is a compact totally disconnected space such that $C(L)$ is isomorphic to $C(K)$.
Our space $K$ is $ccc$ (as it supports a finite measure).
Rosenthal's result from \cite{Ro69} states that this is reflected by the isomorphic structure of $C(K)$,
that property of $K$  is equivalent to saying that every weakly compact subset
of $C(K)$ is separable. It follows that $L$ must be $ccc$ too.

By Theorem \ref{pl15}, $L$ has a $\pi$-base consisting of open sets $V$  such that
$\overline{V}$ is a continuous image of a closed subspace of $K$. By Theorem \ref{prop},
such a set $\overline{V}$ is then metrizable so, in particular, $V$ is separable. It follows that $L$
is a $ccc$ space having a $\pi$-base made of open separable subspaces and this implies
that $L$ is separable itself.

On the other hand, separability of $L$ is in contradiction with Lemma \ref{nonsep}. Indeed, 
if $\{y_n: n<\omega\}$ is a dense subset of $L$
and $T: C(K)\to C(L)$ is an isomorphism then, for every $n$,
\[ C(K)\ni f\to Tf(y_n),\]
 is a continuous functional on $C(K)$ which is represented by
a signed measure $\nu_n$ on $K$. Writing  every $\nu_n$ as $\nu_n=\nu_n^+-\nu_n^-$ (as a difference of two nonegative measures), we get a countable family $\{ \nu_n^+,\nu_n^-: n<\omega\}$
of nonnegative measures separating elements of $C(K)$.
\end{proof}

One can check that if we let $K$ to be connected $L$-space supporting a measure $\mu$ of countable type
then $C(K)^\ast$ has a $weak^\ast$ separable dual unit ball; in particular,
Lemma \ref{nonsep} does not hold. We do not know if $C(K)$ satisfies the assertion of Theorem \ref{final}.
Let us recall that Talagrand's construction mentioned in the introductory section gives, under CH, a tricky space $K$ for which
$C(K)^\ast$ is $weak^\ast$ separable while the dual unit ball in $C(K)^\ast$ is not.
Such an example was later constructed in the usual set theory, see \cite{APR14}.

\end{document}